\documentclass[11pt]{article}

\usepackage{amssymb}
\usepackage{amsthm}
\usepackage{color}

\newtheorem{thm}{Theorem}
\newtheorem{lem}{Lemma}[section]

\newtheorem{rmk}{Remark}

\title{Global dynamics and spectrum comparison of a reaction-diffusion system with mass conservation}

\author{Evangelos Latos, Yoshihisa Morita, and Takashi Suzuki}

\date{}

\begin{document}

\maketitle 

\begin{abstract}{We study global-in-time behavior of the solution to a reaction-diffusion system with mass conservation, as proposed in the study of cell polarity, particularly, the second model of \cite{oi07}.  First, we show global-in-time existence of solution with compact orbit and then we examine stability and instability of stationary solutions.}{Reaction diffusion system; mass conservation; cell polarity; global-in-time behavior; Lyapunov function.}
%%%% If classification number provided then
\\
2000 Math Subject Classification: {\bf MSC2010.} 35K457, 92C37.
\end{abstract}

\begin{flushleft}
{\bf keywords.} reaction diffusion system, mass conservation, cell polarity, global-in-time behavior, Lyapunov function
\end{flushleft}

\section{Introduction}

The purpose of the present paper is to study the mass conserved reaction-diffusion system 
\begin{eqnarray} 
& & u_t=D\Delta u+f(u,v), \quad \tau v_t=\Delta v-f(u,v), \quad \mbox{in $\Omega\times(0,T)$}, \nonumber\\ 
& & \left.\frac{\partial}{\partial \nu}(u,v)\right\vert_{\partial\Omega}=0, \quad \left. (u,v)\right\vert_{t=0}=(u_0(x), v_0(x)), 
 \label{rd0} 
\end{eqnarray} 
where $\Omega\subset{\bf R}^N$ is a bounded domain with smooth boundary $\partial\Omega$, $\nu$ is the outer unit normal vector, $D, \tau>0$ are constants, and $(u_0, v_0)=(u_0(x), v_0(x))$ are smooth nonnegative functions.  Given sufficiently smooth nonlinearity $f=f(u,v)$, standard theory allows the existence of a unique  local-in-time classical solution $(u,v)=(u(\cdot, t), v(\cdot,t))$ to (\ref{rd0}).  Then mass conservation property for this system writes 
\begin{equation} 
\frac{d}{dt}\int_\Omega u+\tau v\ dx=0. 
 \label{tmc}
\end{equation} 

Several equations in this form are used in the study of cell polarity, e.g. \cite{oi07, iom07}. It is expected that different spieces inside the cell shall separate according to their diffusion coefficients, i.e. slow and fast diffusions will localize the spieces near the membrane and in the cytosol, respectively.  Although three kind of molecules are interacting inside the cell in \cite{oi07}, each one of them has two phases, active and inactive which are characterized by slow and fast diffusions, respectively. Problem (\ref{rd0}) thus focuses on these two phases of a single species, ignoring interactions between the other species. 

Therefore, Turing pattern (see \cite{tur52}) is suspected for problem (\ref{rd0}), that is, the appearance of spatially inhomogeneous stable stationary states induced by diffusion. In \cite{oi07} the authors presented the following three models for this purpose,  
\begin{eqnarray} 
& & f(u,v)=-\frac{au}{u^2+b}+v, \nonumber\\ 
& & f(u,v)=-\alpha_1\left[\frac{ u+v}{\biggl(\alpha_2(u+v)+1\biggr)^2}-v\right], \nonumber\\ 
& & f(u,v)=\alpha_1(u+v)[(\alpha u+v)(u+v)-\alpha_2],  
 \label{models}
\end{eqnarray} 
where $a$, $b$, $\alpha,\alpha_1$, and $\alpha_2$ are positive constants.  So far, mathematical analysis is done for the first model, noticing the similarity between the Fix-Caginalp model \cite{st09} (see \cite{jm13, ls13, mor12, mo10}). 

This paper deals with the second form of the reaction term $f(u,v)$ of (\ref{models}). The first fact we shall confirm on this model is the non-negativity of the solution.  
\begin{thm} 
The solution $(u,v)=(u(\cdot,t), v(\cdot,t))$ to (\ref{rd0}) for the second case of $f(u,v)$ in (\ref{models}) satisfies 
\begin{equation} 
u(\cdot,t), \ v(\cdot,t)\geq 0 \quad \mbox{on $\overline{\Omega}$}, \ t\geq 0. 
 \label{eqn:positive}
\end{equation} 
 \label{prop:3}
\end{thm} 
Regarding (\ref{eqn:positive}), we let 
\begin{equation} 
h(z)=-\frac{\alpha_1z}{(\alpha_2z+1)^2}, \quad k=\alpha_1,  
 \label{eqn:5}
\end{equation} 
to write the model as 
\begin{eqnarray} 
& & u_t=D\Delta u+h(u+v)+kv, \nonumber\\ 
& & \tau v_t=\Delta v-h(u+v)-kv, \qquad \ \ \mbox{in $\Omega\times(0,T)$}, \nonumber\\ 
& & \left. \frac{\partial}{\partial \nu}(u,v)\right\vert_{\partial\Omega}=0, \quad \left. (u, v)\right\vert_{t=0}=(u_0(x), v_0(x)).
 \label{rdm1a} 
\end{eqnarray} 
Here we assume $\tau\neq 1$ and furthermore, 
\begin{equation} 
\xi=\frac{1-\tau D}{\tau-1}>0, \quad \alpha=\frac{1-D}{\tau-1}>0, 
 \label{parameter}
\end{equation} 
that is, either $\tau>1>\tau D$ or $\tau D>1>\tau$.  Using 
\begin{eqnarray} 
& & w=Du+v,\quad z=u+v, \nonumber\\ 
& & g(z)=(1-D)h(z)-kDz, 
 \label{eqn:7}
\end{eqnarray} 
system (\ref{rdm1a}) transforms into 
\begin{eqnarray} 
& & z_t=D\Delta z+(w_t-D\Delta w+kw)+g(z), \nonumber\\ 
& & w_t+\xi z_t=\alpha\Delta w, \qquad \qquad \qquad \quad \mbox{in $\Omega\times(0,T)$}, \nonumber\\ 
& & \left. \frac{\partial}{\partial \nu}(z,w)\right\vert_{\partial\Omega}=0, \quad \left. (z, w)\right\vert_{t=0}=(z_0(x), w_0(x)). 
 \label{rdm1} 
\end{eqnarray} 

If the second term on the right-hand side of the first equation of system (\ref{rdm1}) is reduced to $kw$, we obtain 
\begin{eqnarray} 
& & z_t=D\Delta z+kw+g(z), \quad w_t+\xi z_t=\alpha\Delta w \qquad \mbox{in $\Omega\times(0,T)$} \nonumber\\ 
& & \left. \frac{\partial}{\partial \nu}(z,w)\right\vert_{\partial\Omega}=0, \quad \left. (z, w)\right\vert_{t=0}=(z_0(x), w_0(x)), 
 \label{eqn:10}
\end{eqnarray} 
where $z_0=u_0+v_0$ and $w_0=Du_0+v_0$.  It is a generalization of the Fix-Caginalp model \cite{fix83, cag86} for $g(z)=z-z^3$.  We noticed that the first model of (\ref{models}) is reduced to (\ref{eqn:10}) (see \cite{mo10}).  Then, as in the Fix-Caginalp model \cite{st09}, we used a variational structure arising between the Lyapunov function and the stationary state, to clarify the global-in-time dynamics \cite{ls13} in accordance with a spectral property of the stationary state \cite{mor12}.  

Here we show similar properties for problem (\ref{rdm1}). In this model, we still have a Lyapunov function which induces a variational function to formulate a stationary state.  Accordingly the non-stationary solution is global-in-time (Theorem \ref{gethm}), while any local minimum is dynamical stable (Theorem \ref{thm:2}).  Furthermore, the Morse index of the stationary solution is equal to the dynamical instability if $\xi\eta_2>k$, where $\eta_2$ denotes the second eigenvalue of $-\Delta$ under the Neumann boundary condition (Theorem \ref{thm:4}).  

Although (\ref{rdm1}) is derived from (\ref{rdm1a}) for the case of $\tau\neq 1$, system (\ref{rdm1a}) itself has a Lyapunov function even for $\tau=1$. This fact was noticed by \cite{jm13} to confirm the existence of global-in-time solution and the spectral comparison property of stationary solutions.  In the following section we shall confirm that the Lyapunov function of $\tau=1$ used by \cite{jm13} is regarded as a limit case under suitable scaling.  

\section{Summary} 

To begin with, we note that mass conservation (\ref{tmc}) takes the form 
\[ \frac{d}{dt}\int_\Omega \xi z+ w\ dx=0, \] 
in $(z,w)$-variable of (\ref{eqn:7}).  Noticing this property,  we set 
\begin{equation} 
\int_\Omega \xi z+ w\ dx=\int_\Omega \xi z_0+ w_0\ dx=\lambda.
 \label{tmc1}
\end{equation} 

To derive the Lyapunov function of (\ref{rdm1}), we multiply the first equation of (\ref{rdm1}) with $z_t$ to obtain 
\begin{equation} 
\Vert z_t\Vert^2_2+\frac{d}{dt}\int_\Omega\frac{D}{2}\vert \nabla z\vert^2-G(z) \ dx =(w_t-D\Delta w+kw,z_t),
 \label{mfe}
\end{equation} 
where 
\[ G(z)=\int_0^zg(z) \ dz, \] 
and $(\cdot,\cdot )$ denotes the $L^2-$inner product. Multiplying the second equation of (\ref{rdm1}) with $w_t-D\Delta w+kw$, next, we obtain 
\begin{eqnarray} 
& & \xi(z_t,w_t-D\Delta w+kw)=(-w_t+\alpha \Delta w,w_t-D\Delta w+kw) =-\Vert w_t\Vert_2^2 \nonumber\\ 
& & \quad -\alpha D\Vert \Delta w\Vert_2^2 -\alpha k\Vert \nabla w\Vert_2^2 -\frac{d}{dt}\int_\Omega\frac{\alpha+D}{2}\vert\nabla w\vert^2+\frac{k}{2}w^2 \ dx. 
 \label{mse}
\end{eqnarray} 
From (\ref{mfe}) and (\ref{mse}) it follows that 
\begin{eqnarray*} 
& & \xi\Vert z_t\Vert^2_2+\Vert w_t\Vert^2_2+\alpha D\Vert\Delta w\Vert_2^2+\alpha k\Vert \nabla w\Vert_2^2 \\ 
& & \quad =-\frac{d}{dt}\int_\Omega \frac{\alpha+D}{2}\vert \nabla w\vert^2+\frac{k}{2}w^2+\frac{\xi D}{2}\vert \nabla z\vert^2-\xi G(z) \ dx. 
\end{eqnarray*} 
Therefore, 
\begin{equation} 
L(z,w)=\int_\Omega\frac{\alpha+D}{2}\vert \nabla w\vert^2+\frac{k}{2}w^2+\frac{\xi D}{2}\vert \nabla z\vert^2-\xi G(z)\ dx, 
 \label{eqn:lya1}
\end{equation} 
is a Lyapunov function with: 
\begin{eqnarray} 
& & \frac{d}{dt}L(z,w) =-\left\{ \xi\Vert z_t\Vert^2_2+\Vert w_t\Vert^2_2+\alpha D\Vert \Delta w\Vert_2^2+\alpha k\Vert \nabla w\Vert_2^2\right\}\leq 0. 
 \label{eqn:lya2}
\end{eqnarray} 

Now we formulate the stationary state of (\ref{rdm1}).  First, 
\[ \alpha\Delta  w=0, \quad \left. \frac{\partial w}{\partial\nu}\right\vert_{\partial\Omega}=0, \] 
holds in the stationary state of (\ref{rdm1}) and hence $w=w(x)$ is a spatially homogeneous function denoted by $w=\overline{w}\in {\bf R}$.  Then the total mass conservation (\ref{tmc1}) implies 
\begin{equation} 
\lambda=\int_\Omega \xi z+ \overline{w} \ dx, 
 \label{eqn:lambda}
\end{equation} 
hence 
\begin{equation} 
\overline{w}=\frac{1}{\vert \Omega\vert}\left( \lambda-\xi\int_\Omega z\right). 
 \label{eqn:ws}
\end{equation} 
Plugging (\ref{eqn:ws}) into the first equation, we see that the stationary state of (\ref{rdm1}) is reduced to a single equation concerning $z=z(x)$, that is, 
\begin{equation} 
-D\Delta z=g(z)+\frac{k}{\vert\Omega\vert}\left(\lambda-\xi\int_\Omega z\right), \quad \left. \frac{\partial z}{\partial\nu}\right\vert_{\partial\Omega}=0. 
 \label{eqn:13}
\end{equation} 
This problem is the Euler-Lagrange equation concerning the functional 
\begin{equation} 
J_\lambda(z)=\int_\Omega \frac{D}{2}\vert \nabla z\vert^2-G(z)-\frac{k\lambda}{\vert\Omega\vert}z \ dx+\frac{k\xi}{2\vert\Omega\vert}\left( \int_\Omega z\ dx\right)^2
 \label{eqn:functional}
\end{equation} 
defined for $z\in H^1(\Omega)$. 

Our point is to regard (\ref{eqn:13}) from the global dynamics of (\ref{rdm1}). First, the Lyapunov function guarantees the global-in-time solution.  Let $(u_0,v_0)\in X=C^2(\overline{\Omega})^2$ and $E_\lambda$ be the set of solutions $z=z(x)$ to (\ref{eqn:13}) for $\lambda\in {\bf R}$ defined by 
\begin{equation} 
\lambda=\int_\Omega \xi z_0+w_0 \ dx=\int_\Omega u_0+\tau v_0 \ dx. 
 \label{eqn:lambda-1}
\end{equation} 

\begin{thm}
If (\ref{parameter}) holds, the solution $(u,v)=(u(\cdot,t), v(\cdot,t))$ to (\ref{rdm1a}) with (\ref{eqn:5}) is global-in-time.  The orbit ${\cal O}=\{ (u(\cdot,t), v(\cdot,t))\}_{t\geq 0}\subset X$ is compact and hence the $\omega$-limit set defined by 
\begin{eqnarray*} 
\omega(u_0, v_0)=\{ (u_\ast, v_\ast) \mid \mbox{$\exists t_k\uparrow +\infty $} 
\quad \mbox{such that\quad $\Vert (u(\cdot,t_k), v(\cdot,t_k))-(u_\ast, v_\ast)\Vert_X=0$}\}, 
\end{eqnarray*} 
is nonempty, compact, and connected.  Furthermore, any $(u_\ast, v_\ast)\in \omega(u_0,v_0)$ admits $z_\ast\in E_\lambda$ such that 
\begin{equation} 
u_\ast=\frac{w_\ast-z_\ast}{D-1}, \quad v_\ast=\frac{Dz_\ast-w_\ast}{D-1}, 
 \label{eqn:19-1}
\end{equation} 
for $w_\ast\in {\bf R}$ defined by 
\begin{equation} 
w_\ast=\frac{1}{\vert\Omega\vert}\left( \lambda-\xi\int_\Omega z_\ast\right). 
 \label{eqn:20}
\end{equation} 
Finally, it holds that 
\begin{equation} 
\lim_{t\uparrow+\infty}\Vert w(\cdot,t)-\langle w(t)\rangle\Vert_{C^2}=0, 
 \label{eqn:t6}
\end{equation} 
for 
\[ \langle w(t)\rangle=\frac{1}{\vert\Omega\vert}\int_\Omega w(\cdot,t). \] 
 \label{gethm}
\end{thm}

As we have seen, any stationary solution $(u_\ast,v_\ast)$ to (\ref{rdm1}) takes a critical point $z_\ast\in H^1(\Omega)$ of $J_\lambda(z)$ in (\ref{eqn:functional}) through (\ref{eqn:19-1})-(\ref{eqn:20}). Now we examine its dynamical stability.  The first result follows from the {\it semi-unfolding-minimality} property which is valid between the Lyapunov function $L(u,v)$ and variational functional $J_\lambda(v)$.  This structure of the second model is similar to the one of the first model of (\ref{models}) studied in \cite{ls13}.  

\begin{thm} 
Given $0\leq (u_0,v_0)\in X$, let $z_\ast\in H^1(\Omega)$ be a local minimum of $J_\lambda(z)$ in (\ref{eqn:functional}) for $\lambda$ defined by (\ref{eqn:lambda-1}). Then $(u_\ast,v_\ast)$ derived from (\ref{eqn:19-1})-(\ref{eqn:20}) is a dynamically stable stationary state of (\ref{rdm1a}). 
 \label{thm:2}
\end{thm} 

Finally we pay attention to the linearized stability. We write (\ref{rdm1}) as 
\begin{eqnarray} 
& & (1+D\xi/\alpha)z_t-\xi\alpha w_t=D\Delta z+g(z)+kw, \nonumber\\ 
& & w_t+\xi z_t=\alpha\Delta w, \qquad \qquad \qquad \mbox{in $\Omega\times (0,T)$} \nonumber\\ 
& & \left. \frac{\partial}{\partial \nu}(z,w)\right\vert_{\partial\Omega}=0,  
 \label{eqn:22}
\end{eqnarray} 
recalling $1-D/\alpha=\xi/\alpha$.  Then the linearlized equation of (\ref{eqn:22}) around $(z_\ast, w_\ast)$ is given as 
\[ \frac{\partial}{\partial t}M\left( \begin{array}{l} 
Z \\ 
W \end{array} \right)+{\cal A}_1\left( \begin{array}{l} 
Z \\ 
W \end{array}\right)=0, \quad \left. \frac{\partial}{\partial\nu} \left( \begin{array}{l} 
Z \\ 
W \end{array} \right)\right\vert_{\partial\Omega}=0 \] 
where 
\[ M=\left( \begin{array}{cc} 
1+D\xi/\alpha & -\xi/\alpha \\ 
\xi & 1 \end{array} \right), \quad {\cal A}_1=\left( \begin{array}{cc} 
-D\Delta-g'(z_\ast) & -k \\ 
0 & -\alpha\Delta \end{array} \right). \] 
Therefore, the degree of linearized stability of $(z_\ast, w_\ast)$ to (\ref{rdm1}), or equivalently, that of $(u_\ast, v_\ast)$ to (\ref{rdm1a}), is indicated by the number of eigenvalues with negative real parts of the operator ${\cal A}=M^{-1}{\cal A}_1$. This operator is actually realized in $L^2(\Omega; {\bf C})^2$, the Hilbert space composed of square integrable complex-valued functions on $\Omega$, with the domain 
\begin{eqnarray*} 
& & D({\cal A})=\left\{ \left(\begin{array}{l} 
Z \\ 
W \end{array} \right) \in H^2(\Omega; {\bf C})^2 \right. \\ 
& & \quad \left. \mid \int_\Omega W+\xi Z \ dx=0, \ \left. \frac{\partial}{\partial \nu}\left( \begin{array}{l} 
Z \\ 
W \end{array} \right) \right\vert_{\partial\Omega}=0 \right\}. 
\end{eqnarray*} 
 
This $z_\ast$, on the other hand, is also a stationary state of 
\begin{equation} 
z_t=-\delta J_\lambda(z), 
 \label{eqn:gradient}
\end{equation} 
that is, 
\begin{equation} 
z_t=D\Delta z+g(z)+\frac{k}{\vert\Omega\vert}\left(\lambda-\xi\int_\Omega z\right), \quad \left. \frac{\partial z}{\partial\nu}\right\vert_{\partial\Omega}=0.   
 \label{rdm2}
\end{equation} 
The degree of linearized stability of $z_\ast$ to (\ref{rdm2}), on the other hand, is indicated by the number of negative eigenvalues of ${\cal L}$, the self-adjoint operator in $L^2(\Omega)$ defined by 
\begin{equation} 
{\cal L}\varphi=-\left( D\Delta \varphi+g'(z_\ast)\varphi-\frac{k\xi}{\vert\Omega\vert}\int_\Omega \varphi \right), 
 \label{eqn:26}
\end{equation} 
with the domain 
\[ D({\cal L})=\left\{ \varphi\in H^2(\Omega) \mid \left. \frac{\partial\varphi}{\partial\nu}\right\vert_{\partial\Omega}=0 \right\}. \] 

The following theorem assures that these two Morse indices coincide, provided that 
\begin{equation} 
\xi\eta_2>k, 
 \label{eqn:23}
\end{equation} 
recalling that $\eta_2$ is the second eigenvalue of $-\Delta$ with the Neumann boundary condition.  

\begin{thm} 
Any eigenvalue $\sigma\in {\bf C}$ of ${\cal A}$ in $\mbox{Re} \ \sigma< k/2\xi$ is real and is provided with equal algebraic and geometric multiplicities.  If (\ref{eqn:23}) is the case, furthermore, the numbers of negative and zero eigenvalues of ${\cal A}$ and ${\cal L}$ coincide. 
 \label{thm:4} 
\end{thm} 

Theorem \ref{thm:4} is regarded as a spectral comparison property first observed by \cite{bf90}.  It has been examined for the first model of (\ref{models}) by \cite{mor12} and for the second model with $\tau=1$ by \cite{jm13}. Here we use a similar argument as in \cite{cjm} for the proof. 

Concluding this section, we confirm that the Lyapunov function $L(u,v)$ and stationary state valid to $\tau\neq 1$, that is, (\ref{eqn:lya1}) and (\ref{eqn:13}), respectively, are reduced to those for $\tau=1$ used in \cite{jm13}, under suitable scaling. In the following, we assume $D\neq 1$, because $\tau=D=1$ is the trivial case of (\ref{rd0}).  

First, given $\tau\neq 1$, we define $\hat L(z,w; \tau)$ by 
\[ L(z,w)=\xi\hat L(z,w;\tau), \quad \xi=\xi(\tau)=\frac{1-\tau D}{\tau -1}. \] 
Since 
\begin{eqnarray*} 
& & \lim_{\tau\rightarrow 1}\frac{\alpha+D}{\xi}=\lim_{\tau\rightarrow 1}\left\{\frac{1-D}{1-\tau D}+\frac{D(\tau-1)}{1-\tau D}\right\}=1, \\ 
& & \lim_{\tau\rightarrow 1}\frac{k}{\xi}=0, 
\end{eqnarray*} 
it follows that 
\[ \hat L(z,w)\equiv \lim_{\tau\rightarrow 1}\hat L(z,w;\tau)=\int_\Omega \frac{1}{2}\vert \nabla w\vert^2+\frac{D}{2}\vert \nabla z\vert^2-G(z) \ dx, \] 
which is the Lyapunov function used in \cite{jm13}.  

Next, to derive the limit problem of (\ref{eqn:13}) we take 
\[ \hat\lambda=\lambda/\xi=\int_\Omega z+w/\xi \ dx. \] 
By taking $\tau\rightarrow 1$, it holds that 
\begin{equation} 
\hat\lambda=\int_\Omega z. 
 \label{eqn:18}
\end{equation} 
On the other hand, by $\lambda=\xi\hat \lambda$ we write (\ref{eqn:13}) as 
\[ -D\Delta z=g(z)+\frac{k\xi}{\vert \Omega\vert} \left( \hat \lambda-\int_\Omega z\right), \quad \left. \frac{\partial z}{\partial \nu}\right\vert_{\partial\Omega}=0. \] 
Therefore, we can require the limit problem as $\tau\rightarrow 1$ to be 
\begin{equation} 
-D\Delta z=g(z)+\mu, \quad \left. \frac{\partial z}{\partial \nu}\right\vert_{\partial\Omega}=0 
 \label{eqn:30}
\end{equation} 
with some $\mu\in {\bf R}$.  From the solvability of (\ref{eqn:30}) it follows that 
\[ \mu =-\frac{1}{\vert \Omega\vert}\int_\Omega g. \] 
Hence we end up with 
\begin{equation} 
-D\Delta z=g(z)-\frac{1}{\vert\Omega\vert}\int_\Omega g(z), \quad \left. \frac{\partial z}{\partial \nu}\right\vert_{\partial\Omega}=0. 
 \label{eqn:19}
\end{equation} 

The stationary state of (\ref{rdm1a}) with $\tau=1$ is now formulated by (\ref{eqn:18})-(\ref{eqn:19}), using $z=u+v$.  This is the Euler-Lagrange equation of the functional 
\[ \hat J_{\hat \lambda}(z)=\int_\Omega \frac{D}{2}\vert \nabla z\vert^2-G(z) \ dx, \] 
defined for 
\[ H=\{ z\in H^1(\Omega) \mid \int_\Omega z=\hat \lambda\}. \] 

This paper is composed of six sections.  Theorems \ref{prop:3}, \ref{gethm}, \ref{thm:2}, and \ref{thm:4} are proven in sections \ref{sec:3}, \ref{sec:4}, \ref{sec:5}, and \ref{sec:6}, respectively.  

\section{Proof of Theorem \ref{prop:3}}\label{sec:3}

Letting $z_+=\max\{z, 0\}$, we take the auxiliary system 
\begin{eqnarray} 
& & u_t=D\Delta u-\frac{k(u_++v_+)}{(\alpha_1(u_++v_+)+1)^2}+kv_+, \nonumber\\ 
& & \tau v_t=\Delta v+\frac{k(u_++v_+)}{(\alpha_1(u_++v_+)+1)^2} -kv, \quad \mbox{in $\Omega\times(0,T)$}, \nonumber\\ 
& & \left.\frac{\partial}{\partial \nu}(u,v)\right\vert_{\partial\Omega}=0, \quad \left. (u, v)\right\vert_{t=0}=(u_0(x), v_0(x)), 
 \label{eqn:rdm1b}
\end{eqnarray} 
for (\ref{rdm1a}). We shall show the property (\ref{eqn:positive}) for the solution to (\ref{eqn:rdm1b}).  Then this solution solves (\ref{rdm1a}). Hence it coincides with the solution to (\ref{rdm1a}) because of the uniqueness of the latter.  Thus Theorem \ref{prop:3} will be proven. 

In fact, the first equation of (\ref{eqn:rdm1b}) implies 
\[ u_t\geq D\Delta u-ku_+, \quad \left.\frac{\partial u}{\partial \nu}\right\vert_{\partial\Omega}=0, \quad \left. u\right\vert_{t=0}=u_0(x)\geq 0. \] 
Then we obtain $u=u(\cdot,t)\geq 0$ by the maximum principle.   From the second equation of (\ref{eqn:rdm1b}), on the other hand, it holds that 
\[ \tau v_t\geq \Delta v-kv, \quad \left. \frac{\partial v}{\partial \nu}\right\vert_{\partial\Omega}=0, \quad \left. v\right\vert_{t=0}=v_0(x)\geq 0. \] 
Then $v(\cdot,t)\geq0$ follows. 

\begin{rmk} 
Concerning (\ref{rdm1a}), we have $(u,v)=(u(\cdot,t), v(\cdot,t))\geq 0$, provided that (1.3) of \cite{jm13} holds, that is, $h\in C^1[0,\infty)$, $h(s)/s+k\geq 0$ for $s>0$, and $\lim_{s\downarrow 0}h(s)/s=-\beta<0$.  
\end{rmk} 

\section{Proof of Theorem \ref{gethm}}\label{sec:4}

In this section we will prove several a priori estimates. Henceforth, $C_i$, $i=1,2,\cdots, 19$ denote positive constants independent of $t$. 

The first observation is the inequality 
\begin{equation} 
\Vert u(\cdot,t)\Vert_1+\Vert v(\cdot,t)\Vert_1\leq C_1, 
 \label{eqn:l1}
\end{equation} 
which follows from (\ref{tmc1}) and $\xi>0$. Now we show the following lemma. 
\begin{lem} 
It holds that 
\begin{equation} 
\Vert v(\cdot,t)\Vert_\infty\leq C_2. 
 \label{eqn:v-inf}
\end{equation} 
 \label{lem:4}
\end{lem}  
\begin{proof} 
Since 
\begin{equation} 
0\geq h(z)\geq -C_3, \quad z\geq 0, 
 \label{eqn:hu}
\end{equation} 
we have 
\[ \tau v_t\leq \Delta v+C_3-kv, \quad \left. \frac{\partial v}{\partial\nu}\right\vert_{\partial\Omega}=0, \quad \left.v\right\vert_{t=0}=v_0(x)\geq 0, \] 
which implies 
\begin{eqnarray*} 
& & 0\leq v(\cdot,t)\leq \overline{v}(\cdot,t) \\ 
& & \overline{v}(\cdot,t)=e^{t\tau^{-1}(\Delta-k)}v_0+\int_0^te^{(t-t')\tau^{-1}(\Delta-k)}C_3\tau^{-1} \ dt'. 
\end{eqnarray*}  
From the maximum principle in the form of 
\[ \Vert e^{t\Delta}\phi\Vert_\infty\leq \Vert \phi\Vert_\infty, \] 
we obtain 
\[ \Vert \overline{v}(\cdot,t)\Vert_\infty \leq C_4, \] 
and hence (\ref{eqn:v-inf}). 
\end{proof} 

\begin{lem} 
We have 
\begin{equation} 
\Vert z(\cdot,t)\Vert_{H^1}^2+\Vert w(\cdot,t)\Vert_2^2+\int_0^t\Vert z_t(\cdot,t')\Vert_2^2+\Vert \nabla w(\cdot,t')\Vert_2^2 \ dt' \leq C_5. 
 \label{eqn:est1}
\end{equation} 
\end{lem}
\begin{proof}
First, (\ref{eqn:lya2}) implies 
\begin{eqnarray*} 
& & L(z(\cdot,t),w(\cdot,t))+\int_0^t\xi\Vert z_t(\cdot,t')\Vert^2_2+\Vert w_t(\cdot,t')\Vert^2_2 \\ 
& & \quad +\alpha D\Vert \Delta w(\cdot, t')\Vert_2^2+k\alpha\Vert \nabla w(\cdot,t')\Vert_2^2 \ dt' = L(z_0,w_0). 
\end{eqnarray*} 
By (\ref{eqn:7}) and (\ref{eqn:hu}), we have 
\begin{equation} 
g(z)\leq (1+D)C_3, \quad z\geq 0.  
 \label{eqn:33}
\end{equation} 
In (\ref{eqn:lya1}), therefore, it holds that 
\[ G(z)\leq (1+D)C_3z, \quad z\geq 0. \] 
Then (\ref{eqn:l1}) implies (\ref{eqn:est1}). 
\end{proof}

\begin{lem}
It holds that 
\begin{equation} 
\Vert w(\cdot,t)\Vert_\infty\leq C_6. 
 \label{eqn:w2}
\end{equation} 
 \label{lem:6} 
\end{lem}
\begin{proof}
Taking $\mu>0$, we write the second equation of (\ref{rdm1}) as 
\[ w_t=(\alpha\Delta-\mu) w+\mu w-\xi z_t, \quad \left.\frac{\partial w}{\partial \nu}\right\vert_{\partial\Omega}=0, \quad \left. w\right\vert_{t=0}=w_0(x). \] 
Then it follows that 
\begin{equation} 
w(\cdot,t)=e^{t(\alpha\Delta-\mu)}w_0 +\int_0^te^{(t-t')(\alpha\Delta-\mu)}[\mu w(\cdot,t')-\xi z_t(\cdot,t')] \ dt'. 
 \label{eqn:w}
\end{equation} 
To estimate the second term on the right-hand side of (\ref{eqn:w}), we use the semigroup estimate (see \cite{rothe}) 
\begin{equation} 
\Vert e^{t\Delta}\phi\Vert_r\leq C_7(q,r)\max\{ 1, t^{-\frac{N}{2}(\frac{1}{q}-\frac{1}{r})}\}\Vert \phi\Vert_q, \quad 1\leq q\leq r\leq \infty, 
 \label{eqn:sge}
\end{equation} 
recalling that $N$ is the space dimension.  

First, we apply this to $q=2$ and $r=\infty$ for $N=1$ and $1\leq r<\frac{2N}{(N-2)_+}$ for $N\geq 2$.  Then it follows that 
\[ \frac{N}{2}(\frac{1}{2}-\frac{1}{r})<\frac{1}{2}, \] 
and hence 
\begin{eqnarray} 
& & \Vert w(\cdot,t)\Vert_r\leq C_8\Vert w_0\Vert_r +C_8\int_0^t(t-t')^{-\frac{N}{2}(\frac{1}{2}-\frac{1}{r})}e^{-\mu(t-t')}(\Vert w(\cdot,t')\Vert_2 \nonumber\\ 
& & \quad +\Vert z_t(\cdot,t')\Vert_2) \ dt' \leq C_9, 
 \label{eqn:36}
\end{eqnarray} 
from (\ref{eqn:w}).  

If $N\geq 2$ we use also 
\begin{equation} 
\Vert z(\cdot,t)\Vert_r\leq C_{10}, \quad 1\leq r<\frac{2N}{(N-2)_+}, 
 \label{eqn:37}
\end{equation} 
derived from (\ref{eqn:est1}), which implies 
\begin{equation} 
\Vert u(\cdot,t)\Vert_q\leq C_{11}, \quad 1\leq q<\frac{2N}{(N-2)_+}, 
 \label{eqn:38}
\end{equation} 
by (\ref{eqn:v-inf}).  Using (\ref{eqn:hu}), now we have 
\[ u_t\leq D\Delta u+kv, \quad \left.\frac{\partial u}{\partial\nu}\right\vert_{\partial\Omega}=0. \] 
Then it holds that 
\begin{equation} 
0\leq u(\cdot,t)\leq \overline{u}(\cdot,t), 
 \label{eqn:39}
\end{equation} 
for 
\begin{eqnarray*} 
& & \overline{u}(\cdot,t)=e^{(D\Delta-\mu)t}u_0+\int_0^te^{(D\Delta-\mu)(t-t')}[\mu u(\cdot,t')+kv(\cdot,t')] \ dt',  
\end{eqnarray*} 
where the semigroup estimate (\ref{eqn:sge}) is applicable. 

From (\ref{eqn:est1}), (\ref{eqn:38}), and (\ref{eqn:39}) it thus follows that 
\[ \Vert \overline{u}(\cdot,t)\Vert_r\leq C_{12}, \] 
for $1\leq r\leq\infty$ satisfying 
\[ \frac{N}{2}(\frac{1}{q}-\frac{1}{r})<1. \] 
Thus we obtain 
\begin{equation} 
\Vert u(\cdot,t)\Vert_\infty\leq C_{13}, 
 \label{eqn:40}
\end{equation} 
for $N\leq 5$, while (\ref{eqn:38}) is improved as 
\[ \Vert u(\cdot,t)\Vert_q\leq C_{14}, \quad 1\leq q<\frac{2N}{(N-6)_+}, \] 
for $N\geq 6$.  Continuing this procedure, we reach (\ref{eqn:40}) for any $N$ and then (\ref{eqn:w2}) follows from (\ref{eqn:est1}).  
\end{proof}

\begin{proof}[Proof of Theorem \ref{gethm}]
By Lemmas \ref{lem:4} and \ref{lem:6} we have 
\[ \Vert (u(\cdot,t), v(\cdot,t)\Vert_\infty\leq C_{15}. \] 
This implies $T=+\infty$ and also compactness of the orbit 
\[ {\cal O}=\{ (u(\cdot,t), v(\cdot,t))\}_{t\geq 0}\subset C^2(\overline{\Omega})^2.  \] 
From the general theory (see \cite{henry, hale88}) the $\omega$-limit set $\omega(u_0,v_0)$ is non-empty, compact, connected, and invariant under the flow defined by (\ref{rdm1a}), while $L(z, w)$ is constant on $\omega(u_0,v_0)$.  

Given $(u_\ast, v_\ast)\in \omega(u_0, v_0)$, let $(\tilde u, \tilde v)=(\tilde u(\cdot,t), \tilde v(\cdot,t))$ be the solution to (\ref{rdm1a}) for $(u_0,v_0)=(u_\ast, v_\ast)$ and 
\[ \tilde w=D\tilde u+\tilde v, \quad \tilde z=\tilde u+\tilde v.  \] 
From the above property we have 
\[ \frac{d}{dt}L(\tilde z(\cdot,t), \tilde w(\cdot,t))=0, \quad t\geq 0, \] 
and then it follows that 
\[ \tilde z_t=0,\quad\tilde w_t=0,\quad \nabla \tilde w=0,  \] 
from (\ref{eqn:lya2}).  Hence we have 
\[ D\Delta z_\ast+kw_\ast+g(z_\ast)=0, \quad \left. \frac{\partial z_\ast}{\partial\nu}\right\vert_{\partial\Omega}=0, \] 
and $w_\ast\in {\bf R}$.  This $w_\ast$ is determined by the total mass 
\[ \lambda=\int_\Omega \xi z_\ast+w_\ast \ dx, \] 
for $\lambda$ in (\ref{eqn:lambda-1}).  Then (\ref{eqn:20}) follows, and $z=z_\ast$ is a solution to (\ref{eqn:13}).  

Since each $(u_\ast, v_\ast)\in \omega(u_0, v_0)$ satisfies $w_\ast=Du_\ast+v_\ast\in {\bf R}$, it holds that 
\[ \lim_{t\uparrow+\infty}\Vert \nabla w(\cdot,t)\Vert_{C^1}=0. \] 
Then we obtain (\ref{eqn:t6}).  
\end{proof} 

\section{Proof of Theorem \ref{thm:2}}\label{sec:5}

We have derived (\ref{eqn:gradient}) by reducing the second equation of (\ref{rdm1}) to the stationary state.  This process is valid even in the variational level,  that is, between the functionals $L(z,w)$ and $J_\lambda(z)$. In Lemma \ref{prop:7} below, we shall show the semi-unfolding-minimality property, observed in several models in non-equilibrium thermodynamics \cite{is06, is08, st09, st10b, sy07, sy12, pst12} (see also \cite{suz08}).  

For the moment we regard $L(z,w)$ and $J_\lambda(z)$ as smooth functionals of $(z,w)\in H^1(\Omega)\times H^1(\Omega)$ and $z\in H^1(\Omega)$ defined by (\ref{eqn:lya1}) and (\ref{eqn:functional}), respectively.  

\begin{lem} 
Given $\lambda\in{\bf R}$, let $(z,w)\in H^1(\Omega)\times H^1(\Omega)$ satisfy 
\[ \int_\Omega\xi z+w \ dx=\lambda, \] 
and define $\overline{w}\in {\bf R}$ by (\ref{eqn:ws}).  Then it holds that 
\begin{equation}
L(z,w)\geq L(z,\overline{w})=\xi J_\lambda(z)+\frac{\lambda^2 k}{2\vert\Omega\vert}. 
 \label{eqn:sum}
\end{equation} 
 \label{prop:7}
\end{lem} 
\begin{proof}
We have 
\[ \overline{w}=\frac{1}{\vert \Omega\vert}\int_\Omega w, \] 
and hence 
\[ \int_\Omega w^2\geq \int_\Omega \overline{w}^2, \] 
by Jensen's inequality.  Then $L(z,w)\geq L(z,\overline{w})$ follows.  

The second identity of (\ref{eqn:sum}) is now derived as 
\[ L(z,\overline{w}) = \frac{1}{2}\int_\Omega k\overline{w}^2+\xi D|\nabla z|^2-2\xi G(z)\ dx = \xi J_\lambda(z)+\frac{\lambda^2k}{2\vert\Omega\vert}. \] 
\end{proof}

The following lemma holds because $h=h(z)$ is real analytic in $z\geq 0$.  The proof is similar to Lemma 7 of \cite{ls13} and is omitted. 
\begin{lem}  
Let $z_\ast=z_\ast(x)$ be a local minimum of $J_\lambda(z)$, $z\in H^1(\Omega)$, defined by  (\ref{eqn:functional}), where $h=h(z)$ is a real-analytic function of $z\in {\bf R}$. Then there is $\varepsilon_0>0$ such that any $\varepsilon\in (0,\varepsilon_0/4]$ admits $\delta_0>0$ such that 
\begin{equation} 
\Vert z-z_\ast\Vert_{H^1}<\varepsilon_0, \ J_\lambda(z)-J_\lambda(z_\ast)<\delta_0 \quad \Rightarrow \quad \Vert z-z_\ast\Vert_{H^1}<\varepsilon. 
 \label{eqn:ista}
\end{equation} 
 \label{lem:5}
\end{lem} 

We are ready to give the following proof using semi-duality. 

\begin{proof}[Proof of Theorem \ref{thm:2}]
Let $(z_0,w_0)$ be the initial value and let $0\leq z_\ast\in H^1(\Omega)$ be a local minimum of $J_\lambda(z)$, $z\in H^1(\Omega)$, for $\lambda$ defined by (\ref{eqn:lambda}). 

Given $\varepsilon>0$, we shall show the existence of $\delta>0$ such that 
\begin{equation} 
\Vert z_0-z_\ast\Vert_{H^1}+\Vert w_0-\overline{w}\Vert_{H^1}<\delta, 
 \label{eqn:delta}
\end{equation} 
implies 
\begin{equation} 
\Vert z(\cdot,t)-z_\ast\Vert_{H^1}+\Vert w(\cdot,t)-\overline{w}\Vert_{H^1} <C_{16}\varepsilon, \quad t\geq 0, 
 \label{eqn:t7}
\end{equation} 
for $\overline{w}\in {\bf R}$ defined by (\ref{eqn:ws}).  This property will imply the stability of $(z_\ast, \overline{w})$ concerning (\ref{rdm1}) in $X=C^2(\overline{\Omega})^2$, because the orbit 
\[ {\cal O}=\{(u(\cdot,t), v(\cdot,t))\}_{t\geq 0}, \] 
is compact in $X$.  

First, we take $\varepsilon_0>0$ be as in Lemma \ref{lem:5}. Then the total mass conservation in the form of (\ref{tmc1}) implies 
\[ \xi J_\lambda(z(\cdot,t))-\xi J_\lambda(z_\ast)\leq L(z_0,w_0)-L(z_\ast,\overline{w}), \quad t\geq 0, \] 
by (\ref{eqn:lya2}).  Given $\varepsilon\in (0,\varepsilon_0/4]$, next, we take $\delta_0$ as in Lemma \ref{lem:5}. Then we determine $\delta>0$ such that (\ref{eqn:delta}) implies 
\begin{equation} 
\Vert z_0-z_\ast\Vert_{H^1}<\varepsilon_0/2, \quad L(z_0, w_0)-L(z_\ast, \overline{w})<\xi\delta_0. 
 \label{eqn:second}
\end{equation} 
From the second inequality of (\ref{eqn:second}) we have 
\begin{equation} 
J_\lambda(z(\cdot,t))-J_\lambda(z_\ast)<\delta_0, \quad t\geq 0. 
 \label{eqn:t1}
\end{equation} 

Now we show 
\begin{equation} 
\Vert z(\cdot,t)-z_\ast\Vert_{H^1}<\varepsilon_0/2, \quad t\geq 0. 
 \label{eqn:t3}
\end{equation} 
In fact, if this is not the case we have $t_0>0$ such that 
\begin{equation} 
\Vert z(\cdot,t_0)-z_\ast \Vert_{H^1}= \varepsilon_0 /2<\varepsilon_0, 
 \label{eqn:t2}
\end{equation} 
because of the first inequality of (\ref{eqn:second}) and the continuity of $t\mapsto z(\cdot,t)\in H^1(\Omega)$.  Then Lemma \ref{lem:5}, based on (\ref{eqn:t1}) and (\ref{eqn:t2}), implies 
\[ \Vert z(\cdot,t_0)-z_\ast\Vert_{H^1}<\varepsilon\leq \varepsilon_0/4, \] 
a contradiction. Having (\ref{eqn:t1}) and (\ref{eqn:t3}), we obtain 
\begin{equation} 
\Vert z(\cdot,t)-z_\ast\Vert_{H^1}<\varepsilon, \quad t\geq 0.  
 \label{eqn:t5}
\end{equation} 

Since 
\[ \langle w(t)\rangle = \frac{1}{\vert\Omega\vert}\int_\Omega w(\cdot,t)=\frac{1}{\tau\vert\Omega\vert}\left( \lambda-\xi\int_\Omega z\right), \] 
it holds that 
\begin{eqnarray*} 
& & \vert \langle w(t)\rangle-\overline{w}\vert \leq \frac{\xi}{\tau\vert\Omega\vert}\Vert z_\ast-z(\cdot,t)\Vert_1 \\ 
& & \quad \leq \frac{\xi}{\tau\vert\Omega\vert^{1/2}}\Vert z_\ast-z(\cdot,t)\Vert_2 < \frac{\xi\varepsilon}{\tau\vert\Omega\vert^{1/2}}. 
\end{eqnarray*} 
Then (\ref{eqn:t7}) follows from (\ref{eqn:t6}).  
\end{proof} 

\section{Proof of Theorem \ref{thm:4}}\label{sec:6}

The eigenvalue problem of ${\cal A}$ in $L^2(\Omega: {\bf C})^2$ takes the form \[ {\cal A}\left( \begin{array}{c} 
\phi \\ 
\psi \end{array} \right)=\sigma\left( \begin{array}{c}
\phi \\ 
\psi \end{array} \right), \quad \left( \begin{array}{c} 
\phi \\ 
\psi \end{array} \right)\in D({\cal A})\setminus \{0\} \] 
which means $(\phi, \psi)\neq (0,0)$ and 
\begin{eqnarray} 
& & -(D\Delta\phi+g'(z_\ast)\phi+k\psi)=\sigma\{ (1+D\xi/\alpha)\phi-(\xi/\alpha)\psi\}, \nonumber\\ 
& & -\alpha\Delta\psi=\sigma(\psi+\xi\phi), \quad \mbox{in $\Omega$}, \nonumber\\ 
& & \left. \frac{\partial}{\partial \nu}(\phi, \psi)\right\vert_{\partial\Omega}=0, \quad \int_\Omega \psi+\xi\phi \ dx=0.
 \label{eqn:53}
\end{eqnarray} 
Henceforth, $( \ \cdot \ , \ \cdot \ )$ and $\Vert \ \cdot \ \Vert$ indicate the inner product and norm in $L^2(\Omega; {\bf C})^2$, respectively. 

\begin{lem} 
Any eigenvalue $\sigma\in {\bf C}$ of ${\cal A}$ satisfying 
\begin{equation} 
\mbox{Re} \ \sigma <\alpha k/2\xi 
 \label{eqn:55+1}
\end{equation} 
is real. 
 \label{lem:6.1}
\end{lem} 
\begin{proof} 
We may assume $\sigma\neq 0$.  Letting 
\[ \sigma_1=\mbox{Re} \ \sigma, \quad \sigma_2=\mbox{Im} \ \sigma, \quad J_1=\mbox{Re} \ (\phi,\psi), \quad J_2=\mbox{Im} \ (\phi, \psi), \]  
we have 
\begin{eqnarray*} 
& & D\Vert \nabla\phi\Vert^2 -\int_\Omega g'(z_\ast)\vert \phi\vert^2-k(J_1-\imath J_2) \\ 
& & \quad =(\sigma_1+\imath \sigma_2)\{ (1+D\xi/\alpha)\Vert \phi\Vert^2-\xi/\alpha(J_1-\imath J_2)\}, \\ 
& & \alpha\Vert \nabla\psi\Vert^2=(\sigma_1+\imath\sigma_2)\Vert \psi\Vert^2+(\sigma_1+\imath\sigma_2)\xi(J_1+\imath J_2), 
\end{eqnarray*}   
by (\ref{eqn:53}).  Then it follows that 
\begin{eqnarray} 
& & kJ_2=\sigma_2(1+D\xi/\alpha)\Vert \phi\Vert^2-(\xi/\alpha)(\sigma_2J_1-\sigma_1J_2), \nonumber\\ 
& & \alpha\Vert \nabla\psi\Vert^2=\sigma_1\Vert\psi\Vert^2+\xi(\sigma_1J_1-\sigma_2J_2), \nonumber\\ 
& & 0=\sigma_2\Vert \psi\Vert^2+\xi(\sigma_2J_1+\sigma_1J_2). 
 \label{eqn:m11}
\end{eqnarray} 
The last two equalities of (\ref{eqn:m11}) imply 
\begin{equation} 
\alpha\sigma_2\Vert \nabla\psi\Vert^2=-(\sigma_1^2+\sigma_2^2)J_2, 
 \label{eqn:54}
\end{equation} 
while from the first and the third equalities we have 
\begin{equation} 
(\sigma_2/\alpha)\Vert \psi\Vert^2+\sigma_2(1+D\xi/\alpha)\Vert \phi\Vert^2=(k-2\xi\sigma_2/\alpha)J_2.  
 \label{eqn:55}
\end{equation} 
Equalities (\ref{eqn:54})-(\ref{eqn:55}) are reduced to 
\[ (\sigma_2/\alpha)\Vert \psi\Vert^2+\sigma_2(1+D\xi/\alpha)\Vert \phi\Vert^2=-\alpha\sigma_2\frac{k-2\xi\sigma_1/\alpha}{\sigma_1^2+\sigma_2^2}\Vert \nabla\psi\Vert^2. \] 
Thus $\sigma_1<\alpha k/2$ implies $\sigma_2=0$.  
\end{proof}

Henceforth, we define $-\Delta_N$ by $-\Delta_N\phi=-\Delta\phi$, $\phi\in D(-\Delta_N)$, and 
\begin{eqnarray*} 
& & D(-\Delta_N)=\left\{ \phi\in H^2(\Omega)\cap L^2_0(\Omega) \mid \left. \frac{\partial\phi}{\partial\nu}\right\vert_{\partial\Omega}=0\right\}, \\ 
& & L^2_0(\Omega)=\{ \phi \in L^2(\Omega) \mid \int_\Omega\phi=0\}. 
\end{eqnarray*} 
Since 
\[ \int_\Omega(-\Delta\phi)=0, \quad \phi\in D(A), \] 
the operator $-\Delta_N$ is a self-adjoint operator in $L^2_0(\Omega)$.  We put also 
\[ Q\phi=\phi-\langle \phi \rangle, \quad \langle \phi \rangle=\frac{1}{\vert\Omega\vert}\int_\Omega\phi, \] 
for $\phi\in L^2(\Omega)$.  

The proof of the following lemma is similar to that of Lemma 3.3 of \cite{cjm}, although more careful computation is needed. 

\begin{lem} 
The algebraic and geometric multiplicities of the eigenvalue $\sigma$ of ${\cal A}$ in (\ref{eqn:55+1}) coincide. 
 \label{lem:6.2}
\end{lem} 
\begin{proof} 
Let 
\[ ({\cal A}-\sigma I)\left( \begin{array}{c} 
\phi_0 \\ 
\psi_0 \end{array} \right)=0, \quad \left( \begin{array}{c} 
\phi_0 \\ 
\psi_0 \end{array} \right)\in D({\cal A})\setminus \{0\}. \]  
To prove 
\begin{equation} 
\mbox{Ker} \ ({\cal A}-\sigma I)=\mbox{Ker} \ ({\cal A}-\sigma I)^m, \quad m\geq 2, 
 \label{eqn:56}
\end{equation} 
it suffices to show the nonexistence of the solution to 
\begin{equation} 
({\cal A}-\sigma I)\left( \begin{array}{c} 
\phi \\ 
\psi \end{array} \right)=\left( \begin{array}{c} 
\phi_0 \\ 
\psi_0 \end{array}\right), \quad \left(\begin{array}{c} 
\phi \\ 
\psi \end{array} \right)\in D({\cal A}). 
 \label{eqn:57}
\end{equation} 

First, equation (\ref{eqn:56}) yields 
\begin{equation} 
{\cal A}_1\left( \begin{array}{c}
\phi_0 \\ 
\psi_0 \end{array}\right)=\sigma M_1\left( \begin{array}{c} 
\phi_0 \\ 
\psi_0 \end{array} \right), \quad \int_\Omega \xi\phi_0+\psi_0 \ dx=0, 
 \label{eqn:58}
\end{equation} 
and hence 
\[ -\alpha\Delta \psi_0=\sigma(\xi\phi_0+\psi_0)=\sigma(\xi Q\phi_0+Q\psi_0), \] 
from the second component. Multiplying $Q$ to both sides, we obtain 
\begin{equation} 
Q\psi_0=\sigma\xi/\alpha(-\Delta_N-\sigma/\alpha)^{-1}Q\phi_0, \quad \langle \psi\rangle =-\xi\langle \phi\rangle. 
 \label{eqn:59}
\end{equation} 
Then the first component of (\ref{eqn:58}) implies 
\begin{eqnarray} 
& & -D\Delta\phi_0-g'(z_\ast)\phi_0+\{ k\xi-\sigma(1+\xi(D+\xi)/\alpha)\}\langle \phi_0\rangle \nonumber\\ 
& & \quad =\sigma\{ (1+D\xi/\alpha)+\xi/\alpha(k-\sigma\xi/\alpha)(-\Delta_N-\sigma/\alpha)^{-1}\}Q\phi_0. 
 \label{eqn:60}
\end{eqnarray} 

Similarly, (\ref{eqn:57}) implies 
\[ ({\cal A}_1-\sigma M_1)\left( \begin{array}{c} 
\phi \\ 
\psi \end{array} \right)=M_1\left( \begin{array}{c} 
\phi_0 \\ 
\psi_0 \end{array} \right), \] 
and hence 
\begin{eqnarray} 
& & -D\Delta\phi-g'(z_\ast)\phi-\sigma(1+D\xi/\alpha)\phi-(k-\sigma\xi/\alpha))\psi \nonumber\\ 
& & \quad =(1+D\xi/\alpha)\phi_0-(\xi/\alpha)\psi_0, \nonumber\\ 
& & -\alpha\Delta\psi-\sigma\psi-\sigma\xi\phi=\xi\phi_0+\psi_0. 
 \label{eqn:61}
\end{eqnarray} 
From the second equation of (\ref{eqn:61}) it follows that 
\[ Q\psi=\frac{\sigma\xi}{\alpha}(-\Delta_N-\sigma/\alpha)^{-1}Q\phi+\frac{1}{\alpha}(-\Delta_N-\sigma/\alpha)^{-1}(\xi Q\phi_0+Q\psi_0). \] 
Plug this into the first equation of (\ref{eqn:61}).  Then we obtain 
\begin{equation} 
\tilde {\cal L}(\phi)=W, 
 \label{eqn:61.5}
\end{equation} 
where 
\begin{eqnarray} 
& & \tilde{\cal L}(\phi)=-D\Delta\phi-g'(z_\ast)\phi+\{k\xi-\sigma(1+\xi(D+\xi)/\alpha)\}\langle \phi\rangle \nonumber\\ 
& & \quad -\sigma\{ (1+D\xi/\alpha)+(\xi/\alpha)(k-\sigma(\xi/\alpha))B^{-1}\}Q\phi, 
 \label{eqn:62}
\end{eqnarray} 
and 
\begin{eqnarray*} 
& & W=(1+\xi(D+\xi)/\alpha)\langle \phi_0\rangle+(1+D\xi/\alpha)Q\phi_0-(\xi/\alpha)Q\psi_0 \\ 
& & \quad +\frac{\xi}{\alpha}(k-\sigma\xi/\alpha)B^{-1}Q\phi_0+\frac{1}{\alpha}(k-\sigma\xi/\alpha)B^{-1}Q\psi_0, 
\end{eqnarray*} 
using $B=-\Delta_N-\sigma/\alpha$. 

The operator $\tilde{\cal L}$ in (\ref{eqn:62}) is realized as a self-adjoint operator in $L^2(\Omega)$ with the domain $D(\tilde{\cal L})=\{ \phi\in H^2(\Omega) \mid \left. \frac{\partial\phi}{\partial\nu}\right\vert_{\partial\Omega}=0\}$. It holds that $\tilde{\cal L}(\phi_0)=0$ by (\ref{eqn:60}). Hence (\ref{eqn:61.5}) implies 
\begin{equation} 
(W,\phi_0)=0. 
 \label{eqn:64}
\end{equation} 
Here we have 
\begin{eqnarray*} 
& & (W,\phi_0)=(1+\xi(D+\xi)/\alpha)\Vert \langle \phi_0\rangle\Vert^2+(1+D\xi/\alpha)\Vert Q\phi_0\Vert^2 \\ 
& & \quad +\frac{\xi}{\alpha}(k-\sigma\xi/\alpha)(B^{-1}Q\phi_0, Q\phi_0) -\frac{\xi}{\alpha}(Q\psi_0, Q\phi_0) \\ 
& & \quad +\frac{1}{\alpha}(k-\sigma\xi/\alpha)(B^{-1}Q\psi_0, Q\phi_0).  
\end{eqnarray*} 
Since (\ref{eqn:59}), the sum of the last three terms on the right-hand side of the above equality is equal to 
\begin{eqnarray*} 
& & \frac{\xi}{\alpha}(k-\sigma\xi/\alpha)(B^{-1}Q\phi_0, Q\phi_0)-\frac{\xi}{\alpha}\cdot\frac{\sigma\xi}{\alpha}(B^{-1}Q\phi_0, Q\phi_0) \nonumber\\ 
& & \qquad +\frac{1}{\alpha}(k-\sigma\xi/\alpha)((\sigma\xi/\alpha)B^{-1}Q\phi_0, B^{-1}Q\phi_0) \\ 
& & \quad = \frac{\xi}{\alpha}\left\{ (k-2\sigma\xi/\alpha)\Vert B^{-1/2}Q\phi_0\Vert^2+\frac{\sigma}{\alpha}(k-\sigma\xi/\alpha)\Vert B^{-1}Q\phi\Vert_2^2 \right\}. 
\end{eqnarray*} 
Hence it follows that 
\begin{eqnarray} 
& & (W, \phi_0)= (1+\xi(D+\xi)/\alpha)\Vert \langle \phi_0\rangle\Vert^2+(1+D\xi/\alpha)\Vert Q\phi_0\Vert^2 \nonumber\\ 
& & +\frac{\xi}{\alpha}\left\{ (k-2\sigma\xi/\alpha)\Vert B^{-1/2}Q\phi_0\Vert^2+\frac{\sigma}{\alpha}(k-\sigma\xi/\alpha)\Vert B^{-1}Q\phi_0\Vert^2\right\}. 
 \label{eqn:65}
\end{eqnarray} 

From the assumption we have 
\[ k-2\sigma\xi/\alpha>0, \] 
recalling Lemma \ref{lem:6.1}. By this condition the right-hand side on (\ref{eqn:65}) is positive if $\sigma\geq 0$.  If $\sigma<0$, on the other hand, it is obvious that this positivity is satisfied. In any case we have $(W,\phi_0)>0$, a contradiction.  
\end{proof} 

For the proof of Theorem \ref{thm:4}, we write the second equation of (\ref{eqn:53}) as 
\[ \alpha(-\Delta_N-(\sigma/\alpha))Q\psi=\sigma\langle \psi\rangle+\sigma\xi\phi=\sigma\xi Q\phi, \] 
that is, 
\begin{equation} 
Q\psi=\sigma(\xi/\alpha)(-\Delta_N-(\sigma/\alpha))^{-1}Q\phi. 
 \label{eqn:66}
\end{equation} 
Next, the first equation of (\ref{eqn:53}) writes 
\begin{eqnarray*} 
& & -D\Delta\phi-g'(z_\ast)\phi-k(\langle \psi\rangle+Q\psi) \\ 
& & \quad = \sigma\left[ (1+D\xi/\alpha)(\langle \phi\rangle+Q\phi)-(\xi/\alpha)(\langle \psi\rangle+Q\psi)\right] \\ 
& & \quad = \sigma\left[ (1+\xi(D+\xi)/\alpha)\langle \phi\rangle+(1+D\xi/\alpha)Q\phi-(\xi/\alpha)Q\psi\right]. 
\end{eqnarray*} 
Therefore, it holds that 
\begin{eqnarray} 
& & -D\Delta\phi-g'(z_\ast)\phi+k\xi\langle\phi\rangle =\sigma\left[ (1+\xi(D+\xi)/\alpha)\langle\phi\rangle +(1+D\xi/\alpha)Q\phi \right. \nonumber\\ 
& & \qquad \left. - \sigma(\xi/\alpha)^2(-\Delta_N-\sigma/\alpha)^{-1}Q\phi+(k\xi/\alpha)(-\Delta_N-\sigma/\alpha)^{-1}Q\phi \right] \nonumber\\ 
& & \quad =\sigma[(1+\xi(D+\xi)/\alpha)\langle\phi\rangle \nonumber\\ 
& & \qquad +\{ (1+D\xi/\alpha)+(\xi/\alpha)(k-\sigma(\xi/\alpha))(-\Delta -\sigma/\alpha)^{-1}\}Q\phi]. 
 \label{eqn:68}
\end{eqnarray} 

By Lemmas \ref{lem:6.1} and \ref{lem:6.2}, any eigenvalue $\sigma$ of ${\cal A}$ in (\ref{eqn:55+1}) is real and is provided with equal algebraic and geometric multiplicities.  Then it holds that 
\[ \frac{\sigma}{\alpha}<\frac{k}{2\xi}<\frac{k}{\xi}<\eta_2, \] 
by (\ref{eqn:23}). Here we put 
\begin{eqnarray*} 
& & M(s)=(1+\xi(D+\xi)/\alpha)(1-Q) \\ 
& & \quad + \{ (1+D\xi/\alpha)+(\xi/\alpha)(k+s\xi)(-\Delta_N+s)^{-1}Q, 
\end{eqnarray*} 
for each $s>-\eta_2$.  From (\ref{eqn:26}) and (\ref{eqn:68}), the complex number $\sigma$ in (\ref{eqn:55+1}) is an eigenvalue of ${\cal A}$ if and only if it is real, $\sigma/\alpha<\eta_2$, and 
\begin{equation} 
{\cal L}\phi=\sigma M(-\sigma/\alpha)\phi, \quad \phi\in D({\cal L})\setminus\{0\}. 
 \label{eqn:69}
\end{equation} 

To study (\ref{eqn:68}) we fix $s>-\eta_2$ and take the eigenvalue problem 
\begin{equation} 
{\cal L}\phi=\mu M(s)\phi, \quad \phi\in D({\cal L})\setminus \{0\}. 
 \label{eqn:70}
\end{equation} 
If $\Sigma(s)$ denotes the set of eigenvalues $\mu$ of (\ref{eqn:70}), then the relation (\ref{eqn:69}) gives $\sigma\in \Sigma(-\sigma/\alpha)$.  

Problem (\ref{eqn:70}) admits infinite number of eigenvalues, which are real, denoted by 
\[ \mu_1(s)\leq \mu_2(s)\leq \cdots \leq \mu_j(s)\leq \cdots\rightarrow+\infty, \] 
according to their multiplicities.  Then we use the weighted $L^2$ norm $\Vert \ \cdot \ \Vert_s$ defined by 
\[ \Vert u\Vert_s^2=(u,u)_s, \quad (u,v)_s=(M(s)u,v). \] 
In fact, min-max principle is available to define those $\Sigma(s)=\{ \mu_j(s)\}_{j=1}^\infty$ through the Rayleigh quotient (see, e.g. \cite{davis95})
\[ R(\phi, s)=\frac{D\Vert \nabla\phi\Vert^2-(g'(z_\ast)\phi, \phi)+k\xi\Vert \langle \phi\rangle\Vert^2}{\Vert \phi\Vert_s^2}. \] 
Thus, it holds that 
\begin{eqnarray} 
& & \mu_j(s)=\inf\{ \sup_{\phi\in X_j\setminus \{0\}}R(\phi,s) \mid X_j\subset H^1(\Omega), \ \mbox{codim} \ X_j=j-1\} \nonumber\\ 
& & \quad =\inf\{ R(\phi, s) \mid \phi\in H^1(\Omega), \ (\phi, \phi_\ell(s))=0, \ 1\leq \ell\leq j-1\}, 
 \label{eqn:mmp}
\end{eqnarray} 
where $\phi_j(s)$ denotes the eigenfunction of (\ref{eqn:70}) corresponding to $\mu_j(s)$ such that $\Vert \phi_j(s)\Vert_s=1$.  

Let the eigenvalues of $-\Delta_N$ be $\{\eta\}_{\ell=2}^\infty$, 
\[ 0<\eta_2\leq \eta_3\leq \cdots \leq\eta_\ell \leq \cdots \rightarrow+\infty, \] 
and $\{ \phi_\ell\}_{\ell=1}^\infty$ be its $L^2$ ortho-normal eigenfunctions. Then we have 
\begin{eqnarray} 
& & \Vert \phi\Vert_s^2=(1+\xi(D+\xi)/\alpha)\langle \phi\rangle^2 \nonumber\\ 
& & \quad +\sum_{\ell=2}^\infty\{ (1+D\xi/\alpha)+ (\xi/\alpha)(k+s\xi)(\eta_\ell+s)^{-1}\}\vert (\phi, \phi_\ell)\vert^2. 
 \label{eqn:70.5}
\end{eqnarray} 
By (\ref{eqn:23}) we have 
\[ 0\leq \frac{k+s\xi}{\eta_\ell+s}\leq C_{17}, \quad s\geq -k/\xi, \ \ell=2,3,\cdots. \] 
Then it holds that 
\[ C_{18}^{-1}R(\phi)\leq R(\phi, s)\leq C_{18}R(\phi), \quad s\geq -k/\xi, \] 
where 
\[ R(\phi)=\frac{D\Vert \nabla\phi\Vert^2-(g'(z_\ast)\phi, \phi)+k\xi\Vert \langle \phi\rangle\Vert^2}{\Vert \phi\Vert^2}. \] 
Hence the number of zeros and that of negative elements of $\{\mu_j(s)\}_{j=1}^\infty$ are equal to the number of zeros and that of negative eigenvalues of ${\cal L}$, denoted by $m^\ast$ and $m$, respectively.  More precisely, we have 
\begin{equation} 
C_{18}^{-1}\mu_j^\ast \leq \mu_j(s) \leq C_{18}\mu_j^\ast, \quad s\in [-k/\xi, +\infty), 
 \label{eqn:71.5} 
\end{equation}  
for each $j$, where $\mu_j^\ast$ denote the $j$-th eigenvalue of ${\cal L}$.  

From (\ref{eqn:69}), the real number $\sigma$ in $\sigma<\alpha k/\xi$ is an eigenvlue of ${\cal A}$ if and only if 
\begin{equation} 
\mu_j(-\sigma/\alpha)=\sigma, 
 \label{eqn:71}
\end{equation} 
for some $j\geq 1$.  In particular, the number of zero eigenvalues of ${\cal A}$ is equal to that of zero elements of $\{ \mu_j(0)\}_{j=1}^\infty$.  Namely, this number is equal to $m^\ast$. 

Rewriting (\ref{eqn:71}) with $s=-\sigma/\alpha$, on the other hand, we see that the number of negative eigenvalues of ${\cal A}$ is equal to that of $s>0$ such that 
\begin{equation} 
\frac{\mu_j(s)}{s}=-\alpha, 
 \label{eqn:75}
\end{equation} 
for some $j=1,\cdots, m$.  

Here we have 
\begin{equation} 
\frac{\partial}{\partial s}\frac{R(\phi, s)}{s}=-\frac{R(\phi, s)}{s^2}\cdot\frac{1}{\Vert \phi\Vert_s^2}\cdot\frac{\partial}{\partial s}(s\Vert \phi\Vert_s^2),  
 \label{eqn:78}
\end{equation} 
and 
\begin{eqnarray*} 
& & \frac{\partial}{\partial s}(s\Vert \phi\Vert_s^2)=(1+\xi(D+\xi)/\alpha)\langle \phi\rangle^2 \\ 
& & \quad + \sum_{\ell=2}^\infty\{ (1+D\xi/\alpha)+(\xi/\alpha)c_\ell(s)\}\vert (\phi, \phi_\ell)\vert^2, 
\end{eqnarray*} 
with 
\[ 0\leq c_\ell(s)=\frac{\eta_\ell(k+s\xi)}{(\eta_\ell+s)^2}+\frac{\xi s}{\eta_\ell+s}\leq C_{19}, \quad s>0. \] 
Hence it follows that 
\begin{equation} 
C_{19}^{-1}\leq \frac{1}{\Vert \phi\Vert_s^2}\frac{\partial}{\partial s}(s\Vert \phi\Vert_s^2)\leq C_{19}. 
 \label{eqn:79+1}
\end{equation} 
From (\ref{eqn:78}) and (\ref{eqn:79+1}) we have $c_0>0$ independent of $s>0$ and $\phi\in H^1(\Omega)\setminus \{0\}$ such that 
\begin{eqnarray} 
\frac{R(\phi, s')}{s'} & \geq & \frac{R(\phi,s)}{s}-c_0\frac{R(\phi,s)}{s^2}(s'-s)+o(s'-s) \nonumber\\ 
& = & \left( 1-\frac{c_0}{s}(s'-s)\right)\frac{R(\phi,s)}{s}+o(s'-s), 
 \label{eqn:79}
\end{eqnarray} 
as $s'\downarrow s$ uniformly in $s$ and $\phi$. 

By (\ref{eqn:mmp}) and (\ref{eqn:79}) it holds that 
\begin{eqnarray*} 
\frac{\mu_j(s')}{s'} & \geq & \left(1-\frac{c_0}{s}(s'-s)\right)\frac{\mu_j(s)}{s}+o(s'-s) \\ 
& = & \frac{\mu_j(s)}{s}-\frac{c_0\mu_j(s)}{s^2}(s'-s)+o(s'-s), 
\end{eqnarray*} 
as $s'\downarrow s>0$.  In particular, the mapping 
\[ s\in (0,+\infty) \quad \mapsto \quad \frac{\mu_j(s)}{s}<0, \] 
is strictly increasing if $\mu_j(s)<0$, that is, 
\begin{equation} 
\frac{\mu_j(s')}{s'}>\frac{\mu_j(s)}{s}, \quad s'>s>0, 
 \label{eqn:mono-st}
\end{equation} 
for $j=1, \cdots, m$. 

To confirm the continuity of 
\begin{equation} 
s\in (0, +\infty)\mapsto \mu_j(s), 
 \label{eqn:conti}
\end{equation} 
we use its monotonicity (non-increasing) derived from 
\begin{equation} 
\frac{d}{ds}\Vert \phi\Vert_s^2\geq 0, \quad \phi\in L^2(\Omega). 
 \label{eqn:mono-s}
\end{equation} 
Here, we note that (\ref{eqn:mono-s}) is a consequence of (\ref{eqn:23}) and (\ref{eqn:70.5}).  Then (\ref{eqn:mono-st}) and (\ref{eqn:mono-s}) imply 
\[ \mu_j(s_1)\leq \mu_j(s_2)\leq \frac{s_2}{s_1}\mu_j(s_1), \quad 0<s_2\leq s_1, \] 
and hence the continuity of (\ref{eqn:conti}).  

Since (\ref{eqn:71.5}) implies 
\[ \lim_{s\downarrow 0}\frac{\mu_j(s)}{s}=-\infty, \quad \lim_{s\uparrow+\infty}\frac{\mu_j(s)}{s}=0, \] 
each $j=1,\cdots,m$ admits a unique $s=s_j>0$ such that (\ref{eqn:75}).  Thus, the number of negative eigenvalues of ${\cal A}$ is equal to $m$.  The proof is complete. \hfill $\Box$ 
\vskip 0.5cm
\noindent{\bf Acknowledgements.} The first author has partially been supported by NAWI Graz.

\vspace{5mm} 

\begin{flushright} 
Evangelos Latos \\ 
Institute for Mathematics ans Scientific Computing \\ 
University of Graz \\ 
A-8010 Graz, heinrichstr 36 \\ 
Austria \\ 
\vspace{2mm} 
Yoshihisa Morita \\ 
Department of Applied Mathematics and Informatics \\ 
Ryukoku University \\ 
Seta Otsu 520-2194 \\ 
Japan \\ 
\vspace{2mm} 
Takashi Suzuki \\ 
Division of Mathematical Science \\ 
Department of Systems Innovation \\ 
Graduate School of Engineering Science \\ 
Osaka University \\ 
Machikaneyamacho, Toyonakashi, 560-8531 \\ 
Japan 
\end{flushright}

\end{document}